\documentclass[11pt, leqno, twoside]{article}

\usepackage{amssymb}
\usepackage{amsmath}
\usepackage{amsthm}
\usepackage{amsfonts}
\usepackage{mathrsfs}
\usepackage{indentfirst}
\usepackage{graphicx}
\usepackage{txfonts}

\usepackage{anysize}

\usepackage{fancyhdr}

\usepackage{color}
\usepackage{setspace}

\usepackage{hyperref}
\hypersetup{colorlinks=true, linkcolor=blue, citecolor=blue}

\allowdisplaybreaks

\usepackage{txfonts}

\newtheorem{thm}{Theorem}[section]
\newtheorem{prop}[thm]{Proposition}
\newtheorem{lem}[thm]{Lemma}

\theoremstyle{definition}

\numberwithin{equation}{section}

\newcommand{\R}{\mathbb{R}}    
\newcommand{\Rn}{\mathbb{R}^n} 

\newcommand{\essinf}{\operatorname{ess\,inf}}  

\newcommand{\infQ}{\underset{y \in Q}{\essinf}~}

\newcommand{\be}{\beta}

\newcommand{\lt}{\left}
\newcommand{\rt}{\right}

\def\hs{\hspace{0.26cm}}
\def\ls{\lesssim}
\begin{document}
	\title{\bf\Large The Distance from Functions in BMO to BLO
		\footnotetext{\hspace{-0.35cm}
			2020 {\it Mathematics Subject Classification}.
			{42B35, 30H35} \endgraf
			{\it Key words and phrases.} Muckenhoupt weight, BMO, BLO, distance.}}
\author{Hua Huang,
Long Huang\footnote{Corresponding author,\ E-mail: longhuang@gzhu.edu.cn/{\color{red}{June 15, 2026}}}\ \ and Ciqiang Zhuo
	}
	\date{ }
	\maketitle
	
	\vspace{-0.6cm}
	
	\begin{center}
		\begin{minipage}{13cm}
			{\small {\bf Abstract}\quad
				Let BMO and BLO denote the spaces of all locally integrable real-valued functions on $\mathbb{R}^n$ with bounded mean oscillation and bounded lower oscillation, respectively. It is well known that $$L^\infty(\mathbb{R}^n)\subsetneqq {\rm BLO}\subsetneqq {\rm BMO}.$$ In 1978, Garnett and Jones gave distance formulas of $f\in {\rm BMO}$ to $L^\infty(\mathbb{R}^n)$ and recently, Angrisani studied the distance of $f\in {\rm BLO}$ to $L^\infty(\mathbb{R}^n)$.
In this paper, we characterize the distance from any given function $f \in {\rm BMO}$ to BLO via the Muckenhoupt weight class $A_p$ as follows
\begin{center}
dist\,($f$,\ BLO)\,$\sim\inf\left\{\xi\in(0,\infty):\ e^{-\frac f\xi}\in A_p\ \mathrm{for\ some\ }p\in(1,\infty)\right\}$.
\end{center}
Two equivalent representations of this distance are also established in terms of exponential form and the infimum of the constant in a variant of John--Nirenberg inequality, respectively.}
		\end{minipage}
	\end{center}
	
	\vspace{0.2cm}
	
	\section{Introduction}
Let $Q$ denote a cube in the Euclidean space $\Rn$ and $|Q|$ denote the Lebesgue measure of $Q$.
We say a locally integrable real-valued function $f$ on $\Rn$ has bounded mean oscillation, denoted by $f\in {\rm BMO}$,
if
		$$\|f\|_{{\rm BMO}}:=\sup _Q \inf _{c\in \R}  \frac{1}{|Q|} \int_Q |f(x)-c|\,dx<\infty,$$
and we say a locally integrable real-valued function $f$ on $\Rn$ has bounded lower oscillation,
denoted by $f\in {\rm BLO}$,
if
	$$\|f\|_{{\rm BLO}}:=\sup _Q \frac{1}{|Q|} \int_Q [f(x)-\infQ f(y)]\,dx<\infty.$$
Here and thereafter, the suprmum is taken over all cubes $Q\subset \Rn$; see \cite{FJ1961,CRR1980}. Notice that ${\rm BLO}$ is not a vector space and, in \cite{Ko01}, Korey point out
that $f(x) :=-\log|x|\in {\rm BLO}$ but $\log|x|\notin {\rm BLO}$.
However, we still call it space and it is direct to verify that, for any $f\in {\rm BLO}$ and $g\in L^\infty(\Rn)$,
\begin{align}\label{6e12}
f+g\in {\rm BLO}
\end{align}
and $$L^\infty(\Rn)\subsetneqq {\rm BLO}\subsetneqq {\rm BMO}.$$

Applying the Muckenhoupt weight class $A_2$, Garnett and Jones \cite{GJ78} gave comparable upper and lower bounds to the distance
$${\rm dist}\,(f,L^\infty):=\inf_{g\in L^\infty(\Rn)}||f-g||_{{\rm BMO}}$$
for any $f\in {\rm BMO}$. Precisely, Garnett--Jones formula states that there exist two positive constants $c_1$ and $c_2$ such that,
for any $f\in {\rm BMO}$,
$$c_1 \alpha(f)\leq {\rm dist}\,(f,L^\infty)\leq c_2\alpha(f),$$
where
$$\alpha(f):=\inf\left\{ \alpha\in(0,\infty):\ e^{\frac f\alpha}\in A_2\right\}.$$
Recently, in \cite{An17, An24}, Angrisani introduced a new equivalent norm of BLO and hence successfully
described the distance of $f\in {\rm BLO}$ to $L^\infty(\Rn)$
by means of the Muckenhoupt weight class $A_1$: there exist two positive constants $d_1$ and $d_2$ such that,
for any $f\in {\rm BLO}$,
$$d_1 \sigma(f)\leq \inf_{g\in { L^\infty(\Rn)}}\|f-g\|_{{\rm BLO}} \leq d_2\sigma(f),$$
where
$$\sigma(f):=\inf\left\{ \sigma\in(0,\infty):\ e^{\frac f\sigma}\in A_1\right\}.$$
	
Motivated by these results and the embedding $L^\infty(\Rn)\subsetneqq {\rm BLO}\subsetneqq {\rm BMO}$, we raise the question: for any given $f\in {\rm BMO}$, how to characterize its distance from the subspace BLO under the BMO norm? To understand this question, in this paper, we establish three different comparable upper and lower bounds to the distance
 $${\rm dist}\,(f,{\rm BLO}):=\inf_{g\in {\rm BLO}}\|f-g\|_{{\rm BMO}}$$
  for $f\in {\rm BMO}$.

To state the main theorems, we begin with some concepts.
	   We say a locally integrable real-valued function $w$ on $\Rn$ is a \emph{weight}, if $0<w(x)<\infty$ for almost every $x\in \Rn$.	For any $p\in (1,\infty)$, the \emph{Muckenhoupt weight class} $A_p$ is defined to be the set of
all weights $w$ satisfying
    	$$[w]_{A_p}:=\sup_Q\left\{\frac{1}{|Q|}\int_Q w(x)\,dx\right\}\left\{\frac{1}{|Q|}\int_Q w(x)^{-\frac{1}{p-1}}\,dx\right\}^{p-1}<\infty$$
    	with the supremum being taken over all cubes $Q\subset \Rn$, and \emph{Muckenhoupt weight class} $A_1$ is defined to be the set of all weights $w$ satisfying
    	$$[w]_{A_1}:=\inf\left\{c\in(0,\infty):\ M w(x)\leq cw(x)\quad  {\rm for\ a.e.} x\in \Rn \right\}<\infty,$$
here and hereafter,
    	$$M w(x):=\sup\limits_{Q\ni x}\frac{1}{|Q|}\int_Q w(y)\,dy.$$
 For any $f\in {\rm BMO}$, let
\begin{equation*}
\xi(f):=\inf\left\{\xi\in(0,\infty):\ e^{-\frac{f}{\xi}}\in A_p\quad \mathrm{for\ some\ }p\in(1,\infty)\right\}.
\end{equation*}

Then the first result of this paper stated as follows.

     \begin{thm}\label{thm}
	There exist positive constants $k_1$ and $k_2$ such that, for any $f\in {\rm BMO}$,
	$$k_1\xi(f)\leq {\rm dist}(f,{\rm BLO})\leq k_2 \xi(f).$$
    \end{thm}
	
Theorem \ref{thm} is proved in Section \ref{s2}. As applications, in Section \ref{s3}, we establish two equivalent representations of the distance in terms of exponential form and the infimum of the constant in a variant of John--Nirenberg inequality, respectively.

Throughout the paper, the notation $f\lesssim g$ (resp. $f\gtrsim g$) means $f\leq Kg$ (resp. $f\geq Kg$)
for a positive constant $K$ independent of the main parameters, and
$f\sim g$ amounts to $f\ls g\lesssim f$. We also use the symbol $K_{(\alpha,\,\beta,\dots)}$ to denote a positive constant
which depends on the parameters $\alpha,\,\beta,\dots$, but may vary line to line.
For any cube $Q\subset\Rn$, we define the \emph{integral average} of $f$ on $Q$ by setting
$$f_Q:=\frac1{|Q|}\int_Q f(x)\,dx.$$

	\section{Proof of Theorem \ref{thm}\label{s2}}

In this section, we prove Theorem \ref{thm}. We begin with three required lemmas.	
	The following conclusion is a consequence of John--Nirenberg inequality \cite[Lemma 1]{FJ1961}.

	\begin{lem}\label{lem:BMO}
		There exist two positive constants $C_1$ and $C_2$ such that, for any $f\in \rm BMO(\Rn)$ and any cube $Q$ of  $\Rn$,
when $0<\eta<\frac{C_2}{\|f\|_{\rm BMO}}$,
		$$\frac{1}{|Q|}\int_Q e^{\eta|f(x)-f_Q|}dx\leq C_1\eta\lt(\frac{C_2}{\|f\|_{\rm BMO}}-\eta\rt)^{-1}.$$
	\end{lem}
	
	The following Jones' factorization theorem concerning $A_p$ weights is just \cite[Theorem]{PW1980}; see also \cite{cjr83} for a simple proof.
	
	\begin{lem}\label{lem:Jones}
		Let $p\in (1,\infty)$. Then $w\in A_p$ if and only if there exist two weights $w_0,w_1\in A_1$ such that $w=w_0w_1^{1-p}$.
	\end{lem}

The Muckenhoupt weight $A_1$ can be used to characterize the space ${\rm BLO}$; see \cite{CRR1980}.

\begin{lem}\label{lem0514}
$f\in {\rm BLO}$ if and only if $e^{\alpha f}\in A_1$ for some positive constant $\alpha$.
\end{lem}
	
	Now we prove Theorem \ref{thm}.
	
\begin{proof}[Proof of Theorem \ref{thm}]
For $f\in {\rm BMO}$, let
\begin{equation}\label{eq526-a}
	E_{\xi(f)}:=\left\{\xi \in (0,\infty) : e^{-\frac{f}{\xi}}\in A_p\ \mathrm{for\ some\ }p\in(1,\infty)\right\}.
\end{equation}
We next show this theorem by two cases: $\xi(f)\in(0,\infty)$ and $\xi(f)=0$.

\textbf{Case (i)} $\xi(f)\in(0,\infty)$. In this case, we first prove that, for any $f\in {\rm BMO}$,	
	\begin{equation*}
		{\rm dist}(f,{\rm BLO})\lesssim{\xi(f)}.
	\end{equation*}
By the definition of $\xi(f)$,
we know that, for any $k\in\mathbb N$, there exists $\xi_0\in E_{\xi(f)}$ such that
\begin{align}\label{6e12a}
\xi(f)\le \xi_0<\left(1+\frac1k\right)\xi(f),
\end{align}
 and hence $w:=e^{-\frac{f}{\xi_0}}\in A_p$ for some $p\in (1,\infty)$.
By this and Lemma \ref{lem:Jones}, we find that there exist $u,v\in A_1$ such that
	$w=uv^{1-p}$. Thus,
	\begin{align}\label{eq-f}
		f&=-\xi_0\ln w=\xi_0\lt[(p-1)\ln v-\ln u\rt]\\
&=\xi_0\lt[(p-1)\ln v-\ln \frac{u}{M u}\rt]-\xi_0\ln M u.\nonumber
	\end{align}

Since $u\in A_1$, it follows from the Lebesgue differentiation theorem that
	$$\frac{1}{[u]_{A_1}}\leq \frac{u}{M u} \leq 1$$
almost everywhere, which implies that $\ln \frac{u}{M u}\in L^\infty(\Rn)$.
By this, \eqref{6e12} and the fact that $\ln v\in {\rm BLO}$ due to Lemma \ref{lem0514},
we conclude that
\begin{equation}\label{eq0522-a}
h:=\xi_0\left[(p-1)\ln v-\ln \frac{u}{M u}\right]\in {\rm BLO}.
\end{equation}
In addition, by \cite[Theorem]{CRR1980}, we infer that $\ln M u \in  {\rm BMO}$. Combining this, (\ref{eq-f}), \eqref{eq0522-a} and \eqref{6e12a}, we finally conclude that
	\begin{align*}
		{\rm dist}(f,{\rm BLO})&\leq \|f-h\|_{\rm BMO}
		=\xi_0\|\ln M u\|_{\rm BMO}
		\lesssim \xi(f).
	\end{align*}

	Conversely, we show that, for any $f\in {\rm BMO}$,
	\begin{equation*}
		{\rm dist}(f,{\rm BLO})\gtrsim{\xi(f)}.
	\end{equation*}
To this end, we only need to prove that, for any $g\in {\rm BLO}$,
\begin{equation}\label{eq0522-c}
\left(\frac{\|f-g\|_{\rm BMO}}{C_2},\infty\right)\subset E_{\xi(f)},
\end{equation}
where $C_2$ is as in Lemma \ref{lem:BMO}.

To achieve this, taking $\xi \in (\frac{\|f-g\|_{\rm BMO}}{C_2},\infty)$, we next show $\xi \in E_{\xi(f)}$. Indeed, for any $g\in {\rm BLO}$, it is easy to find that, for any cube $Q\subset \Rn$ and almost every $x\in Q$,
	\begin{equation*}
		g_Q-g(x)\leq g_Q-\infQ g(y)=\frac{1}{|Q|} \int_Q [g(x)-\infQ g(y)]\,dx\leq \|g\|_{\rm BLO},
	\end{equation*}
	which implies that
	\begin{equation}\label{eq:414cc}
-g(x)\leq \|g\|_{\rm BLO}-g_Q.
\end{equation}
Let $h:= f-g$. Then $h\in {\rm BMO}$ and, by Lemma \ref{lem:BMO} and $\xi \in (\frac{\|h\|_{\rm BMO}}{C_2},\infty)$, we know that
	\begin{equation*}
		\frac{1}{|Q|}\int_Q e^{\frac1\xi |h(x)-h_Q|}\,dx
\leq C_1\frac1\xi\left(\frac{C_2}{\|h\|_{{\rm BMO}}}-\frac1\xi\right)^{-1},
\quad {\rm for\ any \ cube}\ Q\subset \Rn
	\end{equation*}
	with $C_1,C_2$ being as in Lemma \ref{lem:BMO}.
	By this and \eqref{eq:414cc}, we find that, when $\xi\in (\frac{\|h\|_{\rm BMO}}{C_2},\infty)$,
	\begin{align}\label{eq:423c}
		\frac{1}{|Q|}\int_Q e^{-\frac{f(x)}{\xi}}\,dx
&\leq e^{\frac{\|g\|_{\rm BLO}-g_Q}{\xi}}\frac{1}{|Q|}\int_Q e^{-\frac{h(x)}{\xi}}\,dx\\
		&\leq e^{\frac{\|g\|_{\rm BLO}-f_Q}{\xi}}\frac{1}{|Q|}\int_Q e^{\frac{|h_Q-h(x)|}{\xi}}\,dx\nonumber\\
		&\leq C_1 e^{\frac{\|g\|_{\rm BLO}-f_Q}{\xi}}\frac1\xi\left(\frac{C_2}{\|h\|_{{\rm BMO}}}-\frac1\xi\right)^{-1}<\infty\nonumber.
	\end{align}
For $\xi\in (\frac{\|h\|_{\rm BMO}}{C_2},\infty)$, we choose $p\in (1+\frac{\|f\|_{\rm BMO}}{C_2\xi},\infty)$.
 Then by Lemma \ref{lem:BMO}, we have, for any cube $Q$,
	\begin{align*}
\frac{1}{|Q|}\int_Q e^{\frac{f(x)}{\xi(p-1)}}\,dx
&\leq e^{\frac{f_Q}{\xi(p-1)}}\frac{1}{|Q|}\int_Q e^{\frac{1}{\xi(p-1)}|f(x)-f_Q|}\,dx\\
&\leq C_1e^{\frac{f_Q}{\xi(p-1)}}\frac1{\xi(p-1)}\left(\frac{C_2}{\|f\|_{{\rm BMO}}}-\frac1{\xi(p-1)}\right)^{-1}.
\end{align*}
From this and (\ref{eq:423c}), we deduce that, for any cube $Q\subset \Rn$,
	\begin{align*}\label{eq:423d}
		&\left(\frac{1}{|Q|}\int_Q e^{-\frac{f(x)}{\xi}}\,dx\right)
\left\{\frac{1}{|Q|}\int_Q e^{\frac{f(x)}{\xi(p-1)}}\,dx\right\}^{{p-1}}\\
&\hs\leq \frac{C_1^p}{\xi^p(p-1)^{p-1}}e^{\frac{\|g\|_{\rm BLO}}{\xi}}\left(\frac{C_2}{\|f-g\|_{{\rm BMO}}}-\frac1\xi\right)^{-1}
\left(\frac{C_2}{\|f\|_{{\rm BMO}}}-\frac1{\xi(p-1)}\right)^{-(p-1)}<\infty,
	\end{align*}
which implies that $e^{-\frac{f}{\xi}}\in A_p$ for
$p\in (1+\frac{\|f\|_{\rm BMO}}{C_2\xi},\infty)$.
	Thus $\xi \in E_{\xi(f)}$. Therefore \eqref{eq0522-c} holds true and the proof of Case (i) is completed.

\textbf{Case (ii)} $\xi(f)=0$. In this case, for any $\varepsilon\in(0,\infty)$, there exists $\xi_1 \in (0,\varepsilon)$ such that $w:=e^{-\frac{f}{\xi_1}}\in A_p$ for some $p\in (1,\infty)$.
An argument similar to that used in the proof of Case (i) infers that  there exists $u\in A_1$ such that
	\begin{align*}
		{\rm dist}(f,{\rm BLO})&\leq \xi_1\|\ln M u\|_{\rm BMO}
		\leq \varepsilon\|\ln M u\|_{\rm BMO}.
	\end{align*}
	Thus, by the arbitrariness of $\varepsilon\in(0,\infty)$, we find that 	${\rm dist}(f,{\rm BLO})=0=\xi(f)$.
This finishes the proof of Case (ii)  and hence Theorem \ref{thm}.	
    \end{proof}
	
	\section{Two Equivalent Representations of the Distance\label{s3}}
	
In this section, we establish two different equivalent characterizations of ${\rm dist}(f,{\rm BLO})$ for $f\in {\rm BMO}$.
To this end, for any $f\in {\rm BMO}$, let
	 \begin{align*}
	 	\be(f):=\sup \lt\{\beta\in (0,\infty): \sup_Q \frac{1}{|Q|}\int_Q e^{\beta[f_Q-f(x)]}\,dx<\infty \rt\}
	 \end{align*}
and
\begin{align*}
	 	E_{\beta(f)}:=\lt\{\beta\in (0,\infty): \sup_Q \frac{1}{|Q|}\int_Q e^{\beta[f_Q-f(x)]}\,dx<\infty \rt\}
	 \end{align*}
	 Then we have the following representation of ${\rm dist}(f,{\rm BLO})$.
	
	\begin{prop}\label{pro:01}
		Let $f\in {\rm BMO}$. Then $\xi(f)=\frac{1}{\beta(f)}$ and hence
$$k_1\frac{1}{\beta(f)}\leq {\rm dist}(f,{\rm BLO})\leq k_2\frac{1}{\beta(f)},$$
where $k_1$ and $k_2$ are as in Theorem \ref{thm}.
	\end{prop}
	
	\begin{proof}
		We first prove that $\frac{1}{\beta(f)}\leq \xi(f)$. Indeed, for any given $\xi \in E_{\xi(f)}$ with $E_{\xi(f)}$ as in (\ref{eq526-a}), there exists a constant $p\in (1,\infty)$ such that $e^{-\frac{f}{\xi}}\in A_p$ and hence, for any cube $Q\subset \Rn$,
		\begin{equation}\label{eq:413a}
			\lt(\frac{1}{|Q|}\int_Q e^{-\frac{f(x)}{\xi}}\,dx\rt)\lt(\frac{1}{|Q|}\int_Q e^{\frac{f(x)}{\xi(p-1)}}\,dx\rt)^{p-1}\leq [e^{-\frac f\xi}]_{A_p}.
		\end{equation}
		Since
		$$e^{\frac{1}{|Q|}\int_Q \frac{f(x)}{\xi (p-1)}\,dx}\leq \frac{1}{|Q|}\int_Q e^{\frac{f(x)}{\xi(p-1)}}\,dx$$
		by the Jensen inequality, it follows from (\ref{eq:413a}) that
		$$\frac{1}{|Q|}\int_Qe^{\frac{1}{\xi}[f_Q-f(x)]}\,dx\leq [e^{-\frac f\xi}]_{A_p}<\infty.$$
		Thus, $\frac{1}{\xi}\in E_{\beta(f)}$ and $\frac{1}{\xi}\le {\beta(f)}$.
		Therefore, by the arbitrariness of $\xi \in E_{\xi(f)}$, we obtain $\frac{1}{\beta(f)}\leq \xi(f)$.
		
		Conversely, we show that $\frac{1}{\beta(f)}\geq \xi(f)$. For any given $\beta \in E_{\beta(f)}$, there exists a positive constant $K_{(\beta)}$
		such that, for any cube $Q\subset \Rn$,
		$$\frac{1}{|Q|}\int_Q e^{\beta [f_Q-f(x)]}\,dx\leq K_{(\beta)},$$
		and hence
		\begin{align}\label{eq:414a}
			\frac{1}{|Q|}\int_Q e^{-\beta f(x)}\,dx
			\leq K_{(\beta)} e^{-\beta f_Q}<\infty.
		\end{align}	
Taking $p>1+\frac{\beta \|f\|_{\rm BMO}}{C_2}$, then $\frac\beta{p-1}\in(0,\frac{C_2}{\|f\|_{\rm BMO}})$
		and, by Lemma \ref{lem:BMO} we find that
		\begin{align}\label{eq:414b}
			\frac{1}{|Q|}\int_Q e^{\frac{\beta}{p-1} f(x)}\,dx
&=\frac{1}{|Q|}\int_Q e^{\frac{\beta}{p-1}
				[f(x)-f_Q]} e^{\frac{\beta}{p-1} f_Q}\,dx\\
			&\leq e^{\frac{\beta}{p-1} f_Q} \frac{1}{|Q|}
			\int_Q e^{\frac{\beta}{p-1} |f(x)-f_Q|} \,dx\nonumber\\
			&\ls e^{\frac{\beta}{p-1} f_Q} <\infty\nonumber.
		\end{align}
Combining (\ref{eq:414a}) and (\ref{eq:414b}), we conclude that $e^{-\beta f}\in A_p$ for
$p>1+\frac{\beta \|f\|_{\rm BMO}}{C_2}$.
Thus, $\frac{1}{\beta}\in E_{\xi(f)}$, which means $\xi(f)\le \frac 1\beta$. Therefore, by the arbitrariness of $\beta \in E_{\beta(f)}$,
		$\xi(f)\leq \frac{1}{\beta(f)}$.
		This finishes the proof of Proposition \ref{pro:01}
	\end{proof}
	
In \cite{GJ78}, Garnett--Jones formula also yields that
the distance from the space ${\rm BMO}$ to the space $L^\infty(\Rn)$
was equivalently expressed in terms of the infimum of the constant in the John--Nirenberg inequality.
Inspired by this, we finally give the description of the distance from the space ${\rm BMO}$ to the space
${\rm BLO}$ by the infimum of the constant in a variant of John--Nirenberg inequality.

Recall that $f\in {\rm BMO}$ if and only if there exists $\varepsilon\in(0,\infty)$ and $\lambda_0:=\lambda_0(\varepsilon,f)\in [0,\infty)$ such that
$$\sup_{Q} \frac{1}{|Q|} \lt|\lt\{x\in Q:\ |f(x)-f_Q|>\lambda\rt\}\rt|\leq e^{-\frac{\lambda}{\varepsilon}}$$
whenever $\lambda\in(\lambda_0,\infty)$. This implies that, for $f\in \rm BMO(\Rn)$,
there exists $\varepsilon\in(0,\infty)$ and $\lambda_0:=\lambda_0(\varepsilon,f)\in [0,\infty)$ such that when $\lambda\in(\lambda_0,\infty)$,
	$$\sup_{Q} \frac{1}{|Q|} \lt|\lt\{x\in Q:f_Q-f(x)>\lambda\rt\}\rt|\leq e^{-\frac{\lambda}{\varepsilon}}.$$
Define $E_{\varepsilon(f)}$ to be the set of all such $\varepsilon\in(0,\infty)$
and $\varepsilon(f)$ to be the infimum of all such $\varepsilon\in(0,\infty)$.
Then we prove the following result.
	
	\begin{prop}\label{pro:02}
		Let $f\in \rm BMO(\Rn)$. Then $${\rm dist}(f,{\rm BLO})\sim \varepsilon(f),$$
		where the implicit equivalent positive constants are independent of $f$.
	\end{prop}
	
	\begin{proof}
By Proposition \ref{pro:01}, it is enough to show that, for any $f\in \rm BMO(\Rn)$,
\begin{equation}\label{eq0519-a}
\frac1{\beta(f)}\sim \varepsilon(f).
\end{equation}

		We first prove $\frac{1}{\beta(f)}\lesssim\varepsilon(f)$.
To this end, for any given $\varepsilon \in E_{\varepsilon(f)}$, there exists a constant $\lambda_0\in [0,\infty)$ such that, when $\lambda\in(\lambda_0,\infty)$,
		\begin{equation*}
			\frac{1}{|Q|} \lt|\lt\{x\in Q:f_Q-f(x)>\lambda\rt\}\rt|\leq e^{-\frac{\lambda}{\varepsilon}},
\quad {\rm for\ any\ cube}\ Q\subset \Rn.
		\end{equation*}
		By this, we further know that
		\begin{align*}
			\frac{1}{|Q|}\int_Q e^{\frac{1}{2\varepsilon}[f_Q-f(x)]}\,dx
&=\frac{1}{|Q|}\left\{\int_0^{e^{\frac{\lambda_0}{2\varepsilon}}}\lt|\lt\{x\in Q:e^{\frac{1}{2\varepsilon}[f_Q-f(x)]}>\lambda\rt\}\rt|\,d\lambda
			+\int_{e^{\frac{\lambda_0}{2\varepsilon}}}^{\infty}\cdots \,d\lambda\right\}\\
			&\leq e^{\frac{\lambda_0}{2\varepsilon}}+\frac{1}{|Q|}\int_{\lambda_0}^{\infty}\lt|\lt\{x\in Q: f_Q-f(x)>\lambda\rt\}\rt| \frac{1}{2\varepsilon} e^{\frac{\lambda}{2\varepsilon}}\,d\lambda\\
			&\leq e^{\frac{\lambda_0}{2\varepsilon}}+\frac{1}{2\varepsilon}\int_{\lambda_0}^{\infty}  e^{-\frac{\lambda}{2\varepsilon}}\,d\lambda\\
			&\leq e^{\frac{\lambda_0}{2\varepsilon}}+ e^{-\frac{\lambda_0}{2\varepsilon}}<\infty,
		\end{align*}
		which implies that $\frac{1}{2\varepsilon}\in E_{\beta(f)}$ and hence $\beta(f)\geq  \frac{1}{2\varepsilon}$.
		Therefore, by the arbitrariness of $\varepsilon \in E_{\varepsilon(f)}$, $\frac{1}{\beta(f)}\lesssim\varepsilon(f)$
				
		Conversely, we show that $\frac{1}{\beta(f)}\gtrsim \varepsilon(f)$.
Indeed, for any given $\beta\in E_{\beta(f)}$, there exists a constant $C_{(\beta)}\in(1,\infty)$ such that, for any cube $Q\subset \Rn$,
		\begin{equation}\label{eq:413}
			\frac{1}{|Q|}\int_Q e^{\beta [f_Q-f(x)]}\,dx\leq C_{(\beta)}.
		\end{equation}
		By the Chebyshev inequality and (\ref{eq:413}), we find that, for any $\lambda\in (0,\infty)$,
		\begin{align*}
			\frac{1}{|Q|}\lt|\{x\in Q: f_Q-f(x)>\lambda\}\rt|
			\leq e^{-\beta \lambda}\frac{1}{|Q|}\int_Q e^{\beta[f_Q-f(x)]}\,dx
			\leq C_{(\beta)} e^{-\beta \lambda}.
		\end{align*}
Take $\lambda\in(\frac{2 \ln C_{(\beta)}}{\beta},\infty)$. Then
		$$\frac{1}{|Q|}\lt|\{x\in Q: f_Q-f(x)>\lambda\}\rt|\leq e^{-\frac{\beta}{2} \lambda}.$$
		This infers that $\frac{2}{\beta}\in E_{\varepsilon(f)}$ and
		hence $ \varepsilon(f)\leq \frac{2}{\beta}$.
		Therefore, $\varepsilon(f)\leq \frac{2}{\beta(f)}$ by the arbitrariness $\beta\in E_{\beta(f)}$, which further implies \eqref{eq0519-a}.
		This finishes the proof of Proposition \ref{pro:02}
    \end{proof}


\medskip

\noindent

\textbf{Acknowledgements}
This project is partly supported
by the  Guangdong  Basic  and  Applied  Basic  Research  Foundation  (Grant  No.
2026A1515012216), the Science and Technology Projects of Guangzhou (Grant  No. SL2024A04J00209), the National Natural
Science Foundations of China (Grant  No. 12201139), and the Natural Science Foundation of Hunan province (Grant  No. 2024JJ3023).


\medskip

\noindent Hua Huang and Ciqiang Zhuo
	
\smallskip
	
\noindent MOE-LCSM, School of Mathematics and Statistics,
Hunan Normal University,
Changsha, Hunan 410081, China
	
\smallskip

\noindent{\it E-mails:}
\texttt{huahuang@hunnu.edu.cn}

\noindent\phantom{ {\it E-mails}}
\texttt{cqzhuo87@hunnu.edu.cn}

\medskip
\medskip

\noindent Long Huang (Corresponding author)

\smallskip

\noindent School of Mathematics and Information Science,
Guangzhou University, Guangzhou, 510006, China

\smallskip

\noindent {\it E-mail}: \texttt{longhuang@gzhu.edu.cn}

\end{document}